\documentclass[a4paper, 12pt, oneside]{amsart}

\usepackage[T1]{fontenc}
\usepackage[utf8]{inputenc}
\usepackage{mathptmx}       
\usepackage{helvet}         
\usepackage{courier}        
\usepackage{type1cm}        
\usepackage{graphicx}        

\usepackage{amssymb, amsmath}
\usepackage{enumitem}
\usepackage[DIV=10]{typearea}

\newcommand{\Z}{\mathbb{Z}}
\newcommand{\N}{\mathbb{N}}
\newcommand{\D}{\mathcal{D}}
\newcommand{\G}{\mathcal{G}}

\newcommand{\K}{\mathcal{K}}
\newcommand{\RR}{\mathbb{R}}

\newcommand{\Oo}{\mathcal{O}}

\newcommand{\id}{{\operatorname{Id}}}
\newcommand{\aut}{{\operatorname{Aut}}}
\newcommand{\TSS}{{\overline{X}}}
\newcommand{\tss}{{\overline{\sigma}}}

\theoremstyle{theorem}
\newtheorem{theorem}{Theorem}[section]

\newtheorem{corollary}[theorem]{Corollary}
\newtheorem{lemma}[theorem]{Lemma}

\theoremstyle{definition}

\newtheorem{remark}[theorem]{Remark}

\title{Diagonal-preserving gauge-invariant isomorphisms of graph $C^*$-algebras}

\author{Toke Meier Carlsen}
\address{Department of Science and Technology\\University of the Faroe Islands\\
N\'oat\'un 3\\ FO-100 T\'orshavn\\ Faroe Islands}
\email{toke.carlsen@gmail.com}

\author{James Rout}
\address{School of Mathematics and Applied Statistics\\ University of Wollongong\\
NSW 2522\\ Australia}
\email{jdr749@uowmail.edu.au}

\date{\today}
\thanks{The second author was partially supported by the ARC grant DP150101595, the Simons-Fundation grant 346300 and the Polish Government MNiSW 2015-2019 matching fund.}

\begin{document}
	
\maketitle

\begin{abstract}
	We study graph $C^*$-algebras equipped with generalised gauge actions, and characterise in terms of groupoids and groupoid cocycles when two graph $C^*$-algebras are isomorphic by a diagonal-preserving isomorphism that intertwines the generalised gauge actions. We apply this characterisation to show that two Cuntz--Krieger algebras are isomorphic by a diagonal-preserving isomorphism that intertwines the gauge actions if and only if the corresponding one-sided subshifts are eventually conjugate, and that the stabilisation of two Cuntz--Krieger algebras are isomorphic by a diagonal-preserving isomorphism that intertwines the gauge actions if and only if the corresponding two-sided subshifts are conjugate.
\end{abstract}

\section{Introduction}

It is well-known that the properties of the Cuntz--Krieger algebra $\Oo_A$ of a finite square $\{0,1\}$-matrix $A$ are closely connected to the properties of the one- and two-sided subshifts $X_A$ and $\TSS_A$ of $A$. Cuntz and Krieger proved for example in \cite[Proposition 2.17]{CK} that if $A$ and $B$ are finite square $\{0,1\}$-matrices both satisfying condition (I), and $X_A$ and $X_B$ are conjugate, then there exists a diagonal-preserving $*$-isomorphism between $\Oo_A$ and $\Oo_B$ which intertwines the gauge actions of $\Oo_A$ and $\Oo_B$ (the result actually says a bit more than that); and Cuntz proved in \cite[Theorem 2.3]{C2} that if $A$ and $B$ are finite square $\{0,1\}$-matrices such that $A$ and $B$ and their transposes satisfy condition (I), and $\TSS_A$ and $\TSS_B$ are conjugate, then there exists a diagonal-preserving $*$-isomorphism between the  stabilised Cuntz--Krieger algebras $\Oo_A\otimes\mathcal{K}$ and $\Oo_B\otimes\mathcal{K}$ which intertwines the gauge actions of $\Oo_A\otimes\mathcal{K}$ and $\Oo_B\otimes\mathcal{K}$ (this result was first proved under the additional assumption that $A$ and $B$ are irreducible and aperiodic by Cuntz and Krieger in \cite[Theorem 3.8]{CK}). Cuntz also proved in \cite[Theorem 2.4]{C2} that if $A$ and $B$ are finite square $\{0,1\}$-matrices such that $A$ and $B$ and their transposes satisfy condition (I), and $\TSS_A$ and $\TSS_B$ are flow equivalent, then there exists a diagonal-preserving $*$-isomorphism between the stabilised Cuntz--Krieger algebras $\Oo_A\otimes\mathcal{K}$ and $\Oo_B\otimes\mathcal{K}$ (this result was first proved under the additional assumption that $A$ and $B$ are irreducible and aperiodic by Cuntz and Krieger in \cite[Theorem 4.1]{CK}).

The connection between the properties of $\Oo_A$ and the properties of $X_A$ and $\TSS_A$ was further highlighted in \cite{MM} when Matsumoto and Matui presented, among several other interesting results, a converse to \cite[Theorem 4.1]{CK} by proving that if $A$ and $B$ are finite square $\{0,1\}$-matrices that are irreducible and not permutation matrices, and there is a diagonal-preserving $*$-isomorphism between the stabilised Cuntz--Krieger algebras $\Oo_A\otimes\mathcal{K}$ and $\Oo_B\otimes\mathcal{K}$, then $\TSS_A$ and $\TSS_B$ are flow equivalent \cite[Corollary 3.8]{MM}. To prove this result, Matsumoto and Matui used that the $C^*$-algebra $\Oo_A$ can be constructed as the (reduced) $C^*$-algebra of an étale groupoid and proved that if $A$ and $B$ are finite square $\{0,1\}$-matrices that are irreducible and not permutation matrices then there is a diagonal-preserving $*$-isomorphism between $\Oo_A$ and $\Oo_B$ if and only if the corresponding étale groupoids are isomorphic, and if and only if the one-sided subshifts $X_A$ and $X_B$ are continuously orbit equivalent \cite[Theorem 2.3]{MM}. 

Matsumoto has since then used this aproach to study gauge-invariant isomorphisms of Cuntz--Krieger algebras \cite{MCon,MUni}, and has, among other results, proved a converse to \cite[Proposition 2.17]{CK} in the irreducible case when he proved in \cite[Theorem 1.2]{MCon} that if $A$ and $B$ are finite square $\{0,1\}$-matrices that are irreducible and not permutation matrices, then there is a diagonal-preserving $*$-isomorphism between $\Oo_A$ and $\Oo_B$ that intertwines the gauge actions of $\Oo_A$ and $\Oo_B$ if and only if $X_A$ and $X_B$ are eventually conjugate.

In \cite{KPRR} Kumjian, Pask, Raeburn, and Renault extended the definition of Cuntz--Krieger algebras when they used groupoids to construct $C^*$-algebras from directed graphs that are not assumed to be finite. These graph $C^*$-algebras have since attracted a lot of interest, and by using that graph $C^*$-algebras can be constructed from groupoids, \cite[Theorem 2.3]{MM} has recently been transfered to the setting of graph $C^*$-algebras \cite{AER,BCW}.

In this paper we use groupoids to study diagonal-preserving isomorphisms of graph $C^*$-algebras that intertwine generalised gauge actions. Our main result is Theorem~\ref{thm:1} which characterises when there is a diagonal-preserving $*$-isomorphism between two graph $C^*$-algebras that intertwines two generalised gauge actions. The characterisation is given in terms of cocycle-preserving isomorphisms of the corresponding graph groupoids. We also present a ``stabilised version'' of this result in Theorem~\ref{thm:2}. 

By specialising to ordinary gauge actions we show in Theorem~\ref{thm:3} that there is a diagonal-preserving gauge-invariant $*$-isomorphism between two graph $C^*$-algebras if and only if the two graphs are eventually conjugate, and in Theorem~\ref{thm:4} that if we consider two finite graphs with no sinks or sources, then there is a diagonal-preserving gauge-invariant $*$-isomorphism between the stabilisation of the two graph $C^*$-algebras if and only if the two-sided edge shifts of the two graphs are conjugate.

As corollaries, we prove in Corollary~\ref{cor:2} a converse to \cite[Proposition 2.17]{CK} by generalising \cite[Theorem 1.2]{MCon} to the non-irreducible case, and we prove in Corollary~\ref{cor:3} a converse to \cite[Theorem 2.3]{C2}.

As with \cite[Theorem 2.3]{MM}, the results in this paper can be transfered to the setting of Leavitt path algebras. This will be done in \cite{CR}.

\section{Definitions and notation}

For the benefit of the reader, we recall in this section the definitions of directed graphs and their boundary path spaces, graph groupoids, and graph $C^*$-algebras. Most of this is standard and can be found in many other papers.


\subsection{Directed graphs and their boundary path spaces}

A \emph{directed graph} $E$ is a quadruple $E = (E^0,E^1,r,s)$ consisting of countable sets $E^0$ and $E^1$, and range and source maps $r,s:E^1 \to E^0$. An element of $v \in E^0$ is called a \emph{vertex} and an element of $e \in E^1$ is called an \emph{edge}.

A \emph{path} $\mu$ of length $n$ in $E$ is a sequence of edges $\mu = \mu_1 \dots \mu_n$ such that $r(\mu_i) = s(\mu_{i+1})$ for all $1 \le i \le n-1$. The set of paths of length $n$ is denoted $E^n$. We write $|\mu|$ for the length of $\mu$. The range and source maps extend to paths: $r(\mu) := r(\mu_n)$ and $s(\mu) := s(\mu_1)$. We regard the vertices $v \in E^0$ as paths of length $0$, and set $r(v):=s(v):=v$. For $v \in E^0$ and $n \in \N$, we define $v E^n :=\{ \mu \in E^n: s(\mu) = v\}$. We define $E^* := \bigcup_{n\in\N} E^n$ to be the collection of all paths of finite length. We define $E^0_{\operatorname{reg}} := \{ v \in E^0: vE^1 \text{ is finite and nonempty} \}$ and $E^0_{\operatorname{sing}} := E^0 \setminus E^0_{\operatorname{reg}}$. If $\mu = \mu_1 \dots \mu_m, \nu = \nu_1 \dots \nu_n \in E^*$ with $r(\mu) = s(\nu)$, then we let $\mu \nu := \mu_1 \dots \mu_m \nu_1 \dots \nu_n \in E^*$. A \emph{cycle} (sometimes called a \emph{loop} in the literature) in $E$ is a path $\mu \in E^* \setminus E^0$ such that $s(\mu) = r(\mu)$. An edge $e$ is an \emph{exit} to the cycle $\mu$ if there exists $i$ such that $s(e) = s(\mu_i)$ and $e \not = \mu_i$. A graph is said to satisfy \emph{condition (L)} if every cycle has an exit.

An \emph{infinite path} in $E$ is an infinite sequence $x_1 x_2 \dots$ of edges in $E$ such that $r(x_i) = s(x_{i+1})$ for all $i$. We let $E^\infty$ be the set of all infinite paths in $E$. The source maps extends to $E^\infty$: $s(x) := s(x_1)$. We let $|x|=\infty$ for $x\in E^\infty$. The \emph{boundary path space} of $E$ is the space $\partial E := E^\infty \bigcup \{\mu \in E^*: r(\mu) \in E^0_{\operatorname{sing}} \}.$ If $\mu = \mu_1 \dots \mu_m \in E^*$ and $x = x_1 x_2 \dots \in E^\infty$ with $r(\mu) = s(x)$, then we let $\mu x = \mu_1 \dots \mu_m x_1 x_2 \dots \in E^\infty$. For $v \in E^0$, we define $v \partial E := \{x \in \partial E: s(x)=v \}$.

For $\mu \in E^*$, the \emph{cylinder set} of $\mu$ is the set $Z(\mu) := \{\mu x \in \partial E: x \in r(\mu) \partial E \} \subseteq \partial E$. For $\mu \in E^*$ and a finite subset $F \subseteq r(\mu) E^1$, we define $Z(\mu \backslash F) := Z(\mu) \backslash \Big( \bigcup_{e \in F} Z(\mu e) \Big)$.

The boundary path space $\partial E$ is a locally compact Hausdorff space with the topology given by the basis $\{ Z(\mu \backslash F): \mu \in E^*,\ F \text{ is a finite subset of } r(\mu) \partial E \}$, and each such $Z(\mu \backslash F)$ is compact and open (see \cite[Theorems~2.1 and Theorem~2.2]{Web}).

For $n \in \N$, let $\partial E^{\ge n} := \{x \in \partial E: |x| \ge n\} \subseteq \partial E$. Then $\partial E^{\ge n} = \bigcup_{\mu \in E^n} Z(\mu)$ is an open subset of $\partial E$. Define the \emph{edge shift map} $\sigma_E: \partial E^{\ge 1} \to \partial E$ by $\sigma_E(x_1 x_2 x_3 \dots) = x_2 x_3 \dots$ for $x_1 x_2 x_3 \dots \in \partial E^{\ge 2}$ and $\sigma_E(e) = r(e)$ for $e \in \partial E \bigcap E^1$. Let $\sigma_E^0$ be the identity map on $\partial E$, and for $n \ge 1$, let $\sigma_E^n$ be the $n$-fold composition of $\sigma_E$ with itself. Then $\sigma_E^n$ is a local homeomorphism for all $n \in \N$. When we write $\sigma_E^n(x)$, we implicitly assume that $x \in \partial E^{\ge n}$.

\subsection{Graph groupoids}

The \emph{graph groupoid} of a directed graph $E$ is the locally compact, Hausdorff, \'{e}tale topological groupoid
\[\G_E = \{(x,m-n,y): x,y \in \partial E,\ m,n\in \N, \text{ and } \sigma_E^m(x) = \sigma_E^n(y)\},\] with product $(x,k,y)(w,l,z) := (x,k+l,z)$ if $y=w$ and undefined otherwise, and inverse $(x,k,y)^{-1} := (y,-k,x)$. The topology of $\G_E$ is generated by subsets of the form $$Z(U,m,n,V):= \{(x,m-n,y) \in \G_E: x \in U, y \in V, \sigma_E^m(x)=\sigma_E^n(y)\},$$ where $m,n\in\N$, $U$ is an open subset of $\partial E^{\ge m}$ such that $\sigma_E^m$ is injective on $U$, $V$ is an open subset of $\partial E^{\ge n}$ such that $\sigma_E^n$ is injective on $V$, and $\sigma_E^m(U) = \sigma_E^n(V)$. For $\mu,\nu \in E^*$ with $r(\mu) = r(\nu)$, let $Z(\mu,\nu) := Z(Z(\mu),|\mu|,|\nu|,Z(\nu))$. The map $x \mapsto (x,0,x)$ is a homeomorphism from $\partial E$ to the unit space $\G_E^0$ of $\G_E$. 

In this paper a \emph{cocycle} of a groupoid $\G$ is a groupoid homomorphism $c:\G\to\mathbb{Z}$, where we consider $\mathbb{Z}$ to be a groupoid with product and inverse given by the usual group operations. The function $c_E:\G_E\to\mathbb{Z}$ given by $c_E((x,k,y))=k$ is a continuous cocycle.

All isomorphisms between groupoids considered in this paper are, in addition to preserving the groupoid structure, homeomorphisms.

\subsection{Graph $C^*$-algebras}

The \emph{graph $C^*$-algebra} of a directed graph $E$ is the universal $C^*$-algebra $C^*(E)$ generated by mutually orthogonal projections $\{p_v: v \in E^0\}$ and partial isometries $\{s_e: e \in E^1\}$ satisfying
\begin{enumerate}
\item[(CK1)] $s_e^* s_e = p_{r(e)}$ for all $e \in E^1$;
\item[(CK2)] $s_e s_e^* \le p_{s(e)}$ for all $e \in E^1$;
\item[(CK3)] $p_v = \sum_{e \in vE^1} s_e s_e^*$ for all $v \in E^0_{\operatorname{reg}}$.
\end{enumerate}

There is a strongly continuous action $\lambda^E: \mathbb{T}\to\aut(C^*(E))$, called the \emph{gauge  action}, satisfying $\lambda^E_z(p_v) = p_v$ and $\lambda_z^E(s_e) =  zs_e$ for $z \in \mathbb{T}$, $v\in E^0$ and $e \in E^1$.

We let $s_v := p_v$ for $v \in E^0$, and for $n \ge 2$ and $\mu = \mu_1 \dots \mu_n \in E^n$, we let $s_\mu := s_{\mu_1} \dots s_{\mu_n}$. Then $\operatorname{span} \{s_\mu s_\nu^*: \mu, \nu \in E^*, r(\mu) = r(\nu) \}$ is dense in $C^*(E)$. We define $\D(E)$ to be the closure of $\operatorname{span} \{s_\mu s_\mu^*: \mu \in E^* \}$ in $C^*(E)$. Then $\D(E)$ is an abelian subalgebra of $C^*(E)$, and is isomorphic to the $C^*$-algebra $C_0(\partial E)$. Moreover, $\D(E)$ is a maximal abelian subalgebra of $C^*(E)$ if and only if $E$ satisfies condition $(L)$ (see \cite[Example~3.3]{NR}).

Theorem~3.7 of \cite{Web} shows that there is a unique homeomorphism $h_E$ from $\partial E$ to the spectrum of $\D(E)$ given by \[ h_E(x)(s_\mu s_\mu^*) = \begin{cases} 1 \quad \text{ if } x \in Z(\mu), \\ 0 \quad \text{ if } x \not \in Z(\mu). \end{cases} \]

There is a $*$-isomorphism from the $C^*$-algebra of $\G_E$ to $C^*(E)$ that maps $C_0(\G_E^0)$ onto $\D(E)$ (see \cite[Proposition~2.2]{BCW} and \cite[Proposition~4.1]{KPRR}).

\section{Gauge-invariant isomorphisms of graph $C^*$-algebras and cocycle-preserving isomorphisms of graph groupoids}

In this section, we prove our two main results Theorem~\ref{thm:1} and Theorem~\ref{thm:2}.

Let $E$ be a directed graph, and let $k:E^1\to\mathbb{R}$ be a function. Then $k$ extends to a function $k:E^*\to\mathbb{R}$ given by $k(v)=0$ for $v\in E^0$ and $k(e_1\dots e_n)=k(e_1)+\dots +k(e_n)$ for $e_1\dots e_n\in E^n$, $n\ge 1$. We then get a continuous cocycle $c_k:\G_E\to\mathbb{R}$ given by $c_k((\mu x,|\mu|-|\nu|,\nu x))=k(\mu)-k(\nu)$ and a \emph{generalised gauge action} $\gamma^{E,k}:\mathbb{R}\to\operatorname{Aut}(C^*(E))$ given by $\gamma^{E,k}_t(p_v)=p_v$ for $v\in E^0$ and $\gamma^{E,k}_t(s_e)=e^{ik(e)t}s_e$ for $e\in E^1$. 

If $k(e)=1$ for all $e\in E^1$, then $\gamma^{E,k}_{t}=\lambda^E_{e^{it}}$ for all $t\in\mathbb{R}$, where $\lambda^E$ is the usual gauge action on $C^*(E)$.  

\begin{theorem}\label{thm:1}
	Let $E$ and $F$ be directed graphs and $k:E^1\to\mathbb{R}$ and $l:F^1\to\mathbb{R}$ functions. The following are equivalent.
	\begin{enumerate}
		\item[(1)] There is an isomorphism $\Phi:\G_E\to\G_F$ satisfying $c_l(\Phi(\eta))=c_k(\eta)$ for $\eta\in\G_E$.
		\item[(2)] There is a $*$-isomorphism $\Psi:C^*(E)\to C^*(F)$ satisfying $\Psi(\mathcal{D}(E))=\mathcal{D}(F)$ and $\gamma_t^{F,l}\circ\Psi=\Psi\circ\gamma_t^{E,k}$ for $t\in\mathbb{R}$.
	\end{enumerate}
\end{theorem}

To prove the implication (2) $\implies$ (1) of Theorem~\ref{thm:1}, we need a lemma. We recall the extended Weyl groupoid $\G_{(C^*(E),\D(E))}$ constructed in \cite{BCW} from a graph $C^*$-algebra and its diagonal subalgebra. As defined in \cite{Ren2008}, the \emph{normaliser} of $\D(E)$ is the set 
$$N(\D(E)) := \{n \in C^*(E): ndn^*,\ n^*dn \in \D(E) \text{ for all } d \in \D(E) \}.$$
By \cite[Lemma~4.1]{BCW}, $s_\mu s_\nu^* \in N(\D(E))$ for all $\mu,\nu \in E^*$. For $n \in N(\D(E))$, let $\operatorname{dom}(n) := \{ x \in \partial  E: h_E(x)(n^*n) >0 \}$ and $\operatorname{ran}(n) := \{x \in \partial E: h_E(x)(nn^*) > 0\}$. It follows from \cite[Proposition~4.7]{Ren2008} that, for $n \in N(\D(E))$, there is a unique homeomorphism $\alpha_n: \operatorname{dom}(n) \to \operatorname{ran}(n)$ such that $h_E(x)(n^*dn) = h_E(\alpha_n(x))(d)h_E(x)(n^*n)$ for all $d \in \D(E)$. 

The extended Weyl groupoid $\G_{(C^*(E),\D(E))}$ is the collection of equivalence classes for an equivalence relation on $\{(n,x):n \in N(\D(E)),\ x \in \operatorname{dom}(n) \}$ with the partially-defined product $[(n_1,x_1)][(n_2,x_2)] := [(n_1 n_2,x_2)]$ if $\alpha_{n_2}(x_2)=x_1$ and the inverse operation $[(n,x)]^{-1} := [(n^*,\alpha_n(x))]$ (see \cite[Proposition~4.7]{BCW}). By \cite[Proposition~4.8]{BCW}, the map $\phi_E:(\mu x,|\mu|-|\nu|,\nu x) \to [(s_\mu s_\nu^*,\nu x)]$ is a groupoid isomorphism between $\G_E$ and $\G_{(C^*(E),\D(E))}$.

\begin{lemma}\label{lemma}
	Let $E$ be a directed graph, $k:E^1\to\mathbb{R}$ a function, and $\eta\in\G_E$.
	\begin{enumerate}
		\item[(a)] There exist $n\in N(\mathcal{D}(E))$ and $x\in\operatorname{dom}(n)$ such that $\phi_E(\eta)=[(n,x)]$ and $\gamma^{E,k}_t(n)=e^{ic_k(\eta)t}n$ for all $t\in\mathbb{R}$.
		\item[(b)] Suppose $r\in\mathbb{R}$ and that there are $n'\in N(\mathcal{D}(E))$ and $x'\in\operatorname{dom}(n')$ such that $\phi_E(\eta) = [(n',x')]$ and $\gamma^{E,k}_t(n')=e^{irt}n'$ for all $t\in\mathbb{R}$. Then $r=c_k(\eta)$.
	\end{enumerate}
\end{lemma}

\begin{proof}
	(a): Choose $\mu,\nu\in E^*$ with $r(\mu)=r(\nu)$ such that $\eta\in Z(\mu,\nu)$, and let $n:=s_\mu s_\nu^*$ and $x:=s(\eta)$. Then $x\in\operatorname{dom}(n)$, $\phi_E(\eta)=[(n,x)]$, and $\gamma^{E,k}_t(n)=e^{i(k(\mu)-k(\nu))t}n=e^{ic_k(\eta)t}n$ for all $t\in\mathbb{R}$.
	
	(b): Let $\pi$ denote the isomorphism from $C^*(E)$ to $C^*(\G_E)$ given in \cite[Proposition 2.2]{BCW}. As in \cite{BCW}, we think of $C^*(\G_E)$ as a subset of $C_0(\G_E)$ and define $\operatorname{supp}'(f):=\{\zeta\in\G_E:f(\zeta)\ne 0\}$ for $f\in C_0(\G_E)$. Since $\pi(\gamma^{E,k}_t(s_{\mu} s_{\nu}^*))(\zeta)=e^{ic_k(\zeta)t}\pi(s_{\mu} s_{\nu}^*)(\zeta)$ for $\mu,\nu\in E^*$, $t\in\mathbb{R}$ and $\zeta\in \operatorname{supp}'(\pi(s_\mu s_\nu^*))$, and $\operatorname{span}\{s_\mu s_\nu^*:\mu,\nu\in E^*\}$ is dense in $C^*(E)$, we have $\pi(\gamma^{E,k}_t(n'))(\zeta)=e^{ic_k(\zeta)t}\pi(n')(\zeta)$ for $t\in\mathbb{R}$ and $\zeta\in \operatorname{supp}'(\pi(n'))$. Since $\gamma^{E,k}_t(n')=e^{irt}n'$ for all $t\in\mathbb{R}$, it follows that $\operatorname{supp}'(\pi(n'))\subseteq c_k^{-1}(r)$. It follows that $r=c_k(\eta)$ because if $r\ne c_k(\eta)=k(\mu)-k(\nu)$, then we would have that $\{\eta'\in\operatorname{supp}'(\pi(s_\nu s_\mu^* n')):s(\eta')=r(\eta')=s(\eta)\}=\{(s(\eta),m,s(\eta))\}$ for some $m\in\mathbb{Z}\setminus\{0\}$, from which, together with the definition of the equivalence class $[(n',x')]$ (see \cite[Proposition~4.6]{BCW}), it would follow that $\phi_E(\eta)=[(s_\mu s_\nu^*,s(\eta))] \ne [(n',x')]$.
\end{proof}

\begin{proof}[Proof of Theorem~\ref{thm:1}]

(1) $\implies$ (2): Suppose that $\Phi:\G_E\to\G_F$ is an isomorphism such that $c_l(\Phi(\eta))=c_k(\eta)$ for $\eta\in\G_E$. It then follows from \cite[Proposition 2.2]{BCW} that there is a $*$-isomorphism $\Psi:C^*(E)\to C^*(F)$ satisfying $\Psi(\mathcal{D}(E))=\mathcal{D}(F)$ and $\gamma_t^{F,l}\circ\Psi=\Psi\circ\gamma_t^{E,k}$ for $t\in\mathbb{R}$.

(2) $\implies$ (1): Since $\Psi(\D(E)) = \D(F)$, it follows from \cite[Proposition~4.11]{BCW} that $\Psi$ induces an isomorphism $\psi: \G_{(C^*(E),\D(E))} \to \G_{(C^*(F),\D(F))}$. Define $\Phi := \phi_F^{-1} \circ \psi \circ \phi_E$. Then $\Phi: \G_E \to \G_F$ is a groupoid isomorphism. It remains to check that $\Phi$ is cocyle-preserving. Fix $\eta \in \G_E$. By Lemma~\ref{lemma}(a) there exists $[(n,x)] \in \G_{(C^*(E),\D(E))}$ such that $\phi_E(\eta)=[(n,x)]$ and $\gamma_t^{E,k}(n) = e^{i c_k(\eta) t} n$ for all $t \in \mathbb{R}$. We then have $\phi_F(\Phi(\eta)) = [(\Psi(n),\kappa(x))] \in \G_{(C^*(F),\D(F))}$, where $\kappa$ is a homeomorphism from $\partial E$ onto $\partial F$ (see \cite[Proposition~4.11]{BCW}), and $\gamma_t^{F,l}(\Psi(n)) = e^{i c_k(\eta) t} \Psi(n)$ for all $t \in \mathbb{R}$, so $c_k(\eta)=c_l(\Phi(\eta))$ by Lemma~\ref{lemma}(b).
\end{proof}

Next, we present and prove a ``stabilised version'' of Theorem~\ref{thm:1}. We denote by $\K$ the compact operators on $\ell^2(\N)$, and by $\mathcal{C}$ the maximal abelian subalgebra of $\K$ consisting of diagonal operators.

As in \cite{CRS}, for a directed graph $E$, we denote by $SE$ the graph obtained by attaching a head $\dots e_{3,v} e_{2,v} e_{1,v}$ to every vertex $v \in E^0$ (see \cite[Definition 4.2]{T2}). For a function $k:E^1\to\mathbb{R}$, we let $\bar{k}:SE^1\to\mathbb{R}$ be the function given by $\bar{k}(e)=k(e)$ for $e\in E^1$, and $\bar{k}(e_{i,v})=0$ for $v\in E^0$ and $i=1,2,\dots$.

\begin{theorem}\label{thm:2}
	Let $E$ and $F$ be directed graphs and $k:E^1\to\mathbb{R}$ and $l:F^1\to\mathbb{R}$ functions. The following are equivalent.
	\begin{enumerate}
		\item[(A)] There is an isomorphism $\Phi:\G_{SE} \to\G_{SF}$ satisfying $c_{\bar{l}}(\Phi(\eta))=c_{\bar{k}}(\eta)$ for $\eta\in\G_{SE}$.
		\item[(B)] There is a $*$-isomorphism $\Psi:C^*(E)\otimes\mathcal{K}\to C^*(F)\otimes\mathcal{K}$ satisfying $\Psi(\mathcal{D}(E)\otimes\mathcal{C})=\mathcal{D}(F)\otimes\mathcal{C}$ and $(\gamma_t^{F,l} \otimes\id_\mathcal{K})\circ\Psi=\Psi\circ(\gamma_t^{E,k}\otimes\id_\mathcal{K})$ for $t\in\mathbb{R}$. 
	\end{enumerate}
\end{theorem}

\begin{proof}
	By \cite[Theorem~4.2]{T2}, there is an isomorphism $\rho: C^*(E) \otimes \K \to C^*(SE)$. Let $\{\theta_{i,j}\}_{i,j\in\mathbb{N}}$ be the canonical generators of $\mathcal{K}$. For $v\in E^0$ and $i\in\mathbb{N}$ denote by $\mu_{i,v}$ the path $e_{i,v}e_{i-1,v}\dots e_{1,v}$ in $SE$ (if $i=0$, then we let $\mu_{i,v}=v$). One can check that the isomorphism $\rho: C^*(E) \otimes \K \to C^*(SE)$ can be chosen such that $\rho(p_v \otimes \theta_{i,j}) = s_{\mu_{i,v}} s_{\mu_{j,v}}^*$ for $v \in E^0$ and $\rho(s_e \otimes \theta_{i,j}) = s_{\mu_{i,s(e)}} s_e s_{\mu_{j,r(e)}}^*$ for $e \in E^1$. Routine calculations then show that $\rho(\mathcal{D}(E) \otimes \mathcal{C}) = \mathcal{D}(SE)$ and that $\rho \circ (\gamma_t^{E,k} \otimes \id_\K) = \gamma^{SE,\bar{k}}_t\circ\rho$ for $t \in \RR$. The equivalence (A) $\iff$ (B) now follows from the equivalence (1) $\iff$ (2) of Theorem~\ref{thm:1} applied to the graphs $SE$ and $SF$ and the functions $\bar{k}$ and $\bar{l}$.
\end{proof}

\section{Eventual conjugacy of graphs}

Let $E$ and $F$ be directed graphs. Following \cite{MCon}, we say that $E$ and $F$ are \emph{eventually conjugate} if there exists a homeomorphism $h:\partial E\to\partial F$ and continuous maps $k:\partial E^{\ge 1}\to\mathbb{N}$ and $k':\partial F^{\ge 1}\to\mathbb{N}$ such that $\sigma_F^{k(x)}(h(\sigma_E(x)))=\sigma_F^{k(x)+1}(h(x))$ for all $x\in \partial E^{\ge 1}$ and $\sigma_E^{k'(y)}(h^{-1}(\sigma_F(y)))=\sigma_E^{k'(y)+1}(h^{-1}(y))$ for all $y\in \partial F^{\ge 1}$. We call such a homeomorphism $h: \partial E \to \partial F$ an \emph{eventual conjugacy}.

Notice that if $h:\partial E\to\partial F$ is a conjugacy in the sense that $\sigma_F(h(x))=h(\sigma_E(x))$ for all $x\in \partial E^{\ge 1}$, then $h$ is an eventual conjugacy (in this case we can take $k$ and $k'$ to be constantly equal to 0).

\begin{theorem}\label{thm:3}
	Let $E$ and $F$ be directed graphs. The following are equivalent.
	\begin{enumerate}
		\item[(i)] $E$ and $F$ are eventually conjugate.
		\item[(ii)] There is an isomorphism $\Phi:\G_E\to\G_F$ satisfying $c_F(\Phi(\eta))=c_E(\eta)$ for $\eta\in\G_E$.
		\item[(iii)] There is a $*$-isomorphism $\Psi:C^*(E)\to C^*(F)$ satisfying $\Psi(\mathcal{D}(E))=\mathcal{D}(F)$ and $\lambda_z^{F}\circ\Psi=\Psi\circ\lambda_z^{E}$ for $z\in\mathbb{T}$.
	\end{enumerate}
\end{theorem}

\begin{proof}
	(ii) $\iff$ (iii): Let $k:E^1\to\mathbb{R}$ and $l:F^1\to\mathbb{R}$ both be constantly equal to 1. Then $c_k=c_E$, $c_l=c_F$, and $\gamma^{E,k}_{t}=\lambda^E_{e^{it}}$ and $\gamma^{F,l}_{t}=\lambda^F_{e^{it}}$ for all $t\in\mathbb{R}$. It therefore follows from Theorem~\ref{thm:1} that (ii) and (iii) are equivalent.
	
	(i) $\implies$ (ii): Suppose $h:\partial E\to\partial F$ is an eventual conjugacy. Then the map $(x,n,y)\mapsto (h(x),n,h(y))$ is an isomorphism $\Phi:\G_E\to\G_F$ satisfying $c_F(\Phi(\eta))=c_E(\eta)$ for $\eta\in\G_E$.
	
	(ii) $\implies$ (i): Suppose $\Phi:\G_E\to\G_F$ is an isomorphism such that $c_F(\Phi(\eta))=c_E(\eta)$ for $\eta\in\G_E$. Then the restriction of $\Phi$ to $\G_E^0$ is a homeomorphism onto $\G_F^0$. Since the map $x\mapsto (x,0,x)$ is a homeomorphism from $\partial E$ onto $\G_E^0$, and $y\mapsto (y,0,y)$ is a homeomorphism from $\partial F$ onto $\G_F^0$, it follows that there is a a homeomorphism $h:\partial E\to\partial F$ such that $\Phi((x,0,x))=(h(x),0,h(x))$ for all $x\in\partial E$. Since $c_F(\Phi(\eta))=c_E(\eta)$ for all $\eta\in\G_E$, it follows that $\Phi((x,n,y))=(h(x),n,h(y))$ for all $(x,n,y)\in\G_E$. Let $e\in E^1$. Then $\Phi(Z(e,r(e)))$ is an open and compact subset of $c_F^{-1}(1)$. It follows that there exist an $n$, mutually disjoint open subsets $U_1,\dots,U_n$ of $\partial F$, mutually disjoint open subsets $V_1,\dots,V_n$ of $\partial F$, and $k_1,\dots,k_n\in\mathbb{N}$ such that $\Phi(Z(e,r(e)))=\bigcup_{i=1}^n Z(U_i,k_i+1,k_i,V_i)$. Define $k_e:Z(e)\to\mathbb{N}$ by $k_e(x)=k_i$ for $x\in h^{-1}(U_i)$. Then $k_e$ is continuous and $\sigma_F^{k_e(x)}(h(\sigma_E(x)))=\sigma_F^{k_e(x)+1}(h(x))$ for $x\in Z(e)$. By doing this for each $e\in E^1$, we get a continuous map $k:\partial E^{\ge 1}\to\mathbb{N}$ such that $\sigma_F^{k(x)}(h(\sigma_E(x)))=\sigma_F^{k(x)+1}(h(x))$ for all $x\in \partial E^{\ge 1}$.
	
	A continuous map $k':\partial F^{\ge 1}\to\mathbb{N}$ such that $\sigma_E^{k'(y)}(h^{-1}(\sigma_F(y)))=\sigma_E^{k'(y)+1}(h^{-1}(y))$ for all $y\in \partial F^{\ge 1}$ can be constructed in a similar way. Thus, $h$ is an eventual conjugacy.
\end{proof}

Let $A$ be a finite square $\{0,1\}$-matrix, and assume that every row and every column of $A$ is nonzero. As in \cite{CK}, we denote by $\Oo_A$ the Cuntz--Krieger algebra of $A$ with gauge action $\lambda^A$ and canonical abelian subalgebra $\D_A$, and by $(X_A,\sigma_A)$ the one-sided subshift of $A$ (if $A$ does not satisfy condition (I), then we let $\Oo_A$ denote the universal Cuntz--Krieger algebra $\mathcal{A}\Oo_A$ introduced in \cite{aHR}). As in \cite{MCon}, we say that $(X_A,\sigma_A)$ and $(X_B,\sigma_B)$ are eventually one-sided conjugate if there is a homeomorphism $h:X_A\to X_B$ and continuous maps $k:X_A\to\mathbb{N}$ and $k':X_B\to\mathbb{N}$ such that $\sigma_B^{k(x)}(h(\sigma_A(x)))=\sigma_B^{k(x)+1}(h(x))$ for all $x\in X_A$, and $\sigma_A^{k'(y)}(h^{-1}(\sigma_B(y)))=\sigma_A^{k'(y)+1}(h^{-1}(y))$ for all $y\in X_B$.

We obtain from Theorem~\ref{thm:3} the following corollary which was proved in the irreducible case by Kengo Matsumoto in \cite[Theorem 1.2]{MCon}, and which can be seen as a kind of a converse to \cite[Proposition 2.17]{CK}.

\begin{corollary}\label{cor:2}
	Let $A$ and $B$ be finite square $\{0,1\}$-matrices, and assume that every row and every column of $A$ and $B$ is nonzero. There is a $*$-isomorphism $\Psi:\Oo_A\to\Oo_B$ satisfying $\Psi(\D_A)=\D_B$ and $\lambda^B_z\circ\Psi=\Psi\circ\lambda^A_z$ for all $z\in\mathbb{T}$ if and only if $(X_A,\sigma_A)$ and $(X_B,\sigma_B)$ are eventually one-sided conjugate.
\end{corollary}

\begin{proof}
	Let $E_A$ be the graph of $A$, i.e., $E_A^0$ is the index set of $A$, $E_A^1=\{(i,j)\in E_A^0\times E_A^0: A(i,j)=1\}$, and $r((i,j))=j$ and $s((i,j))=i$ for $(i,j)\in E_A^1$. Then $\partial E_A=E_A^\infty$, and there is a homeomorphism from $X_A$ to $E_A^\infty$ that intertwines $\sigma_A$ and $\sigma_{E_A}$. It follows that $(X_A,\sigma_A)$ and $(X_B,\sigma_B)$ are eventually one-sided conjugate if and only if $E_A$ and $E_B$ are eventually one-sided conjugate. 
	
	It is well-known that there is a $*$-isomorphism $\Psi:\Oo_A\to C^*(E_A)$ satisfying $\Psi(\D_A)=\D(E_A)$ and $\lambda^{E_A}_z\circ\Psi=\Psi\circ\lambda^A_z$ for all $z\in\mathbb{T}$. The corollary therefore follows from the equivalence of (i) and (iii) in Theorem~\ref{thm:3}.
\end{proof}

\begin{remark} Kengo Matsumoto has strengthened \cite[Theorem 1.2]{MCon} in \cite{MUni} and shown that if the matrices $A$ and $B$ in Corollary \ref{cor:2} are irreducible and not permutation matrices, then the two conditions in Corollary \ref{cor:2} are equivalent to several other interesting conditions, for example to the condition that there is a $*$-isomorphism $\Psi:\Oo_A\to\Oo_B$ satisfying $\Psi(\D_A)=\D_B$ and $\Psi(\mathcal{F}_A)=\mathcal{F}_B$, where $\mathcal{F}_A$ is the fixed point algebra of $\lambda^A$ and $\mathcal{F}_B$ is the fixed point algebra of $\lambda^B$. \end{remark}

\section{Conjugacy of two-sided shifts of finite type}

For a finite directed graph $E$ with no sinks or sources, we define $\TSS_E$ to be the two-sided edge shift 
$$\TSS_E := \{ (x_n)_{n\in\Z}: x_n \in E^1\text{ and } r(x_n) = s(x_{n+1})\text{ for all }n\in\mathbb{Z}\}$$
equipped with the induced topology of the product topology of $(E^1)^{\mathbb{Z}}$ (where each copy of $E^1$ is given the discrete topology), and let $\tss_E:\TSS_E\to\TSS_E$ be the homeomorphism given by $(\tss_E(x))_m=x_{m+1}$ for $x=(x_n)_{n\in\Z}\in\TSS_E$. 

If $E$ and $F$ are finite directed graphs with no sinks or sources, then a \emph{conjugacy} from $\TSS_E$ to $\TSS_F$ is a homeomorphism $\phi:\TSS_E\to\TSS_F$ such that $\tss_F\circ\phi=\phi\circ\tss_E$. The shift spaces $\TSS_E$ and $\TSS_F$ are said to be \emph{conjugate} if there is a conjugacy from $\TSS_E$ to $\TSS_F$. 

Recall that for a directed graph $E$, we denote by $SE$ the graph obtained by attaching a head $\dots e_{3,v} e_{2,v} e_{1,v}$ to every vertex $v \in E^0$ (see \cite[Definition 4.2]{T2}). Define a function $k_E: (SE)^1 \to \RR$ by $k_E(e) = 1$ for $e \in E^1$ and $k_E(e_{i,v}) = 0$ for $v \in E^0$ and $i=1,2,\dots$. 

\begin{theorem}\label{thm:4}
	Let $E$ and $F$ be directed graphs. The following two conditions are equivalent.
	\begin{enumerate}
		\item[(I)] There is an isomorphism $\Phi:\G_{SE} \to\G_{SF}$ satisfying $c_{k_F}(\Phi(\eta))=c_{k_E}(\eta)$ for $\eta\in\G_{SE}$.
		\item[(II)] There is a $*$-isomorphism $\Psi:C^*(E)\otimes\mathcal{K}\to C^*(F)\otimes\mathcal{K}$ satisfying $ \Psi(\mathcal{D}(E)\otimes\mathcal{C})=\mathcal{D}(F)\otimes\mathcal{C}$ and $(\lambda_z^F \otimes\id_\mathcal{K})\circ\Psi=\Psi\circ(\lambda^E_z\otimes\id_\mathcal{K})$ for $z\in\mathbb{T}$.
	\end{enumerate}
	If $E$ and $F$ are finite graphs with no sinks or sources, then (I) and (II) are equivalent to the following condition.
	\begin{enumerate}
		\item[(III)] The two-sided edge shifts $\TSS_E$ and $\TSS_F$ are conjugate.
	\end{enumerate}
\end{theorem}

\begin{proof}
	An argument similar to the one used in the proof of Theorem~\ref{thm:3} shows that the equivalence of (I) and (II) follows by applying Theorem~\ref{thm:2} to the functions $k:E^1\to\mathbb{R}$ and $l:F^1\to\mathbb{R}$ that are constanly equal to 1.

It remains to establish (I) $\iff$ (III). Suppose that $E$ and $F$ are finite graphs with no sinks or sources. Then $\partial E=E^\infty$ and $\partial F=F^\infty$. As in the proof of Theorem~\ref{thm:2}, for $v\in E^0$ and $i\in\mathbb{N}$, we denote by $\mu_{i,v}$ the path $e_{i,v}e_{i-1,v}\dots e_{1,v}$ in $SE$ (if $i=0$, then we let $\mu_{i,v}=v$). In the proof of \cite[Lemma~4.1]{CRS}, it was shown that there is a homeomorphism $(SE)^\infty \to E^\infty \times \N$ satisfying $\mu_{i,s(x)} x \to (x, i)$ for $x \in E^\infty$ and $i \in \N$. We identify $(SE)^\infty$ with $E^\infty \times \N$.

(III) $\implies$ (I): Suppose $\TSS_E$ and $\TSS_F$ are conjugate. Then there is a conjugacy $\phi: \TSS_E \to \TSS_F$ and an $l \in \N$ such that if $x, x' \in X_E$ with $x_n = x_n'$ for all $n \ge 0$, then $(\phi(x))_n = (\phi(x'))_n$ for all $n \ge 0$, and if $y,y' \in X_F$ with $y_n = y_n'$ for all $n \ge 0$, then $(\phi^{-1}(y))_n = (\phi^{-1}(y'))_n$ for all $n \ge l$. It follows that there is a continuous map $\pi: E^\infty \to F^\infty$ such that $(\pi((x_k)_{k\in\mathbb{N}}))_n=(\phi(x))_n$ for $x=(x_k)_{k\in\mathbb{Z}}\in\TSS_E$ and $n\in\mathbb{N}$. Then $\pi$ is surjective, $\pi \circ \sigma_E = \sigma_F \circ \pi$, and if $\pi(x) = \pi(x')$ for $x, x' \in E^\infty$, then $\sigma_E^l(x)=\sigma_E^l(x')$.

For an infinite path $x=(x_n)_{n\in\mathbb{N}}$ and $k\in\mathbb{N}$, we write $x_{[0,k)}$ for the finite path $x_0x_1\dots x_{k-1}$ of length $k$. It follows from the continuity of $\pi$ and the compactness of $E^\infty$ that we can choose $L\ge l$ such that $x_{[0,L)}=x'_{[0,L)}\implies \pi(x)_{[0,l)}=\pi(x')_{[0,l)}$. Define an equivalence relation $\sim$ on $E^L$ by $\mu\sim\nu$ if there are $x\in Z(\mu)$ and $x'\in Z(\nu)$ such that $\pi(x) = \pi(x')$ (that $\sim$ is transitive follows from the fact that if $\mu,\nu,\eta\in E^L$, $x,x'\in E^\infty$, and $\pi(\mu x)=\pi(\nu x)$ and $\pi(\nu x')=\pi(\eta x')$, then $\pi(\mu x')=\pi(\eta x')$). Then $\pi(x) = \pi(x')$ if and only if $x_{[0,L)}\sim x'_{[0,L)}$ and $\sigma_E^l(x)=\sigma_E^l(x')$.

For each equivalence class $B\in E^L/\sim$ choose a partition $\{A_\mu: \mu \in B \}$ of $\N$ and bijections $f_\mu: A_\mu \to \N$. The map $\psi: (x,n) \mapsto (\pi (x), f^{-1}_{x_{[0,L)}} (n))$ is then a homeomorphism from $(SE)^\infty \to (SF)^\infty$. It is routine to check that 
$$\Phi:\left((x,n),k,(x',n')\right) \mapsto \left(\psi(x,n),k+n'+f^{-1}_{x_{[0,L)}} (n)-n-f^{-1}_{x'_{[0,L)}} (n'),\psi(x',n')\right)$$ 
is a groupoid isomorphism from $\G_{SE}$ to $\G_{SF}$ satisfying $c_{k_F}(\Phi(\eta)) = c_{k_E}(\eta)$ for $\eta \in \G_{SE}$.

(I) $\implies$ (III): Suppose $\Phi: \G_{SE} \to \G_{SF}$ is an isomorphism satisfying $c_{k_F}(\Phi(\eta)) = c_{k_E}(\eta)$ for $\eta \in \G_{SE}$. For $x \in E^\infty$, we have $(x,0) \in (SE)^\infty$ and $((x,0), 0, (x,0)) \in \G_{SE}$. Since $\Phi$ is an isomorphism, we have $\Phi((x,0),0,(x,0)) = ((y,m),0,(y,m))$ for some uniquely determined $y \in F^\infty$ and $m \in \N$. Define $\psi: E^\infty \to F^\infty$ by $\psi(x) := y$. Since $\Phi$ is continuous, the map $(x,0) \mapsto (y,m)$ is continuous, so $\psi$ is also continuous.

We have $((x,0), 1, (\sigma_E(x),0)) \in \G_{SE}$ for $x \in E^\infty$. By the cocyle condition there exist $m,m' \in \N$ such that $\Phi((x,0), 1, (\sigma_E(x),0)) = ((\psi(x),m), 1+m-m', (\psi(\sigma_E(x)), m')) \in \G_{SE}$. Hence there exists $l \in \N$ such that $\sigma_F^{l+1} (\psi(x)) = \sigma_F^l(\psi(\sigma_E(x)))$; let $l(x)$ denote the smallest such number. We check that $l: E^\infty \to \N$ is continuous. Suppose $(x_n)_{n\in \N}$ in $E^\infty$ converges to $x$. Then $((x_n,0),1,(\sigma_E(x_n),0)) \to ((x,0), 1, (\sigma_E(x),0))$. Since $\Phi$ is continuous, there are $(m_n)_{n\in \N}$ and $(m_n')_{n\in \N}$ in $\N$ such that $((\psi(x_n),m_n),1+m_n-m'_n,(\psi(\sigma_E(x_n)),m_n')) \to ((\psi(x),m),1+m-m',(\psi(\sigma_E(x)),m'))$. It follows that $l(x_n) \to l(x)$, so $l: E^\infty \to \N$ is continuous. 

Since $E^\infty$ is compact, it follows that there is an $L \in \N$ such that $\sigma_E^{L+1}(\psi(x)) = \sigma_F^L(\psi(\sigma_E(x)))$ for all $x \in E^\infty$. Define $\varphi:= \sigma_F^L \circ \psi: E^\infty \to F^\infty$. Then $\varphi$ is continuous and satisfies $\varphi \circ \sigma_E = \sigma_F \circ \varphi$. 

For $x=(x_n)_{n\in\mathbb{Z}}$ in $\TSS_E$ or $\TSS_F$ and $k\in\mathbb{Z}$, let $x_{[k,\infty)}$ denote the infinite path $x_kx_{k+1}\dots $. Define $\overline \varphi: \TSS_E \to \TSS_F$ by $\overline (\varphi(x))_{[k,\infty)}=\varphi(x_{[k,\infty)})$ for $x \in \TSS_E$ and $k \in \Z$. Since $\varphi \circ \sigma_E = \sigma_F \circ \varphi$, it follows that $\overline \varphi$ is well-defined. It is routine to check that $\overline \varphi$ is continuous and that $\overline \varphi \circ \tss_E = \tss_F \circ \overline \varphi$. We will next show that $\overline\varphi$ is also bijective. It will then follow that $\overline\varphi$ is a conjugacy and thus that $\TSS_E$ and $\TSS_F$ are conjugate.

We first show that $\overline\varphi$ is injective. Suppose $x=(x_n)_{n\in\mathbb{N}},\ x'=(x'_n)_{n\in\mathbb{N}}\in E^\infty$ and $\varphi(x)=\varphi(x')$. Choose $m,m'\in\mathbb{N}$ such that $\Phi((x,0),0,(x,0))=((\psi(x),m),0,(\psi(x),m))$ and $\Phi((x',0),0,(x',0))=((\psi(x'),m'),0,(\psi(x'),m'))$. Since 
$\sigma_F^L(\psi(x))=\sigma_F^L(\psi(x')),$
it follows that $((\psi(x),m),m-m',(\psi(x'),m'))\in\G_{SF}$ and thus that 
$$((x,0),0,(x',0))=\Phi^{-1}(((\psi(x),m),m-m',(\psi(x'),m')))\in\G_{SE}.$$ 
It follows that there is a $k\in\mathbb{N}$ such that $\sigma_E^k(x)=\sigma_E^k(x')$. Let $k((x,x'))$ be the smallest such $k$. An argument similiar to the one used to prove that $l:E^\infty\to\N$ is continuous, shows that $k:\{(x,x')\in E^\infty\times E^\infty:\varphi(x)=\varphi(x')\}\to\mathbb{N}$ is continuous. Since $\{(x,x')\in E^\infty\times E^\infty:\varphi(x)=\varphi(x')\}$ is closed in $E^\infty\times E^\infty$ and thus compact, it follows that there exists $K\in\mathbb{N}$ such that $\sigma_E^K(x)=\sigma_E^K(x')$ for all $x,x'\in E^\infty$ satisfying $\varphi(x)=\varphi(x')$. The injectivity of $\overline\varphi$ easily follows.

Next, we show that $\overline\varphi$ is surjective. Suppose $y\in F^\infty$. Then 
$$\Phi^{-1}((y,0),0,(y,0))=((x,n),0,(x,n))$$ 
for some $x\in E^\infty$ and some $n\in\mathbb{N}$. Choose $m\in\mathbb{N}$ such that 
$$\Phi((x,0),0,(x,0))=((\psi(x),m),0,(\psi(x),m)).$$ 
Since $((x,0),-n,(x,n))\in\G_{SE}$ and $\Phi((x,0),-n,(x,n))=((\psi(x),m),m,(y,0))$, it follows that there is an $h\in\mathbb{N}$ such that $\sigma_F^h(\psi(x))=\sigma_F^h(y)$. An argument similar to the one used in the previous paragraph, then shows that there is an $H\in\mathbb{N}$ such that for each $y\in F^\infty$ there is an $x\in E^\infty$ such that $\sigma_F^H(\psi(x))=\sigma_F^H(y)$. The surjectivity of $\overline\varphi$ easily follows.
\end{proof}

Let $A$ be a finite square $\{0,1\}$-matrix, and assume that every row and every column of $A$ is nonzero. As in \cite{CK}, we denote by $(\TSS_A,\tss_A)$ the two-sided subshift of $A$. It follows from \cite[Theorem 2.3]{C2} that if $A$ and $B$ are finite square $\{0,1\}$-matrices such that $A$ and $B$ and their transpose satisfy condition (I), and $\TSS_A$ and $\TSS_B$ are conjugate, then there exists a $*$-isomorphism $\Psi:\Oo_A \otimes\mathcal{K}\to \Oo_B \otimes\mathcal{K}$ such that $\Psi(\mathcal{D}_{A}\otimes\mathcal{C})=\mathcal{D}_{B}\otimes\mathcal{C}$ and $(\lambda^B_z\otimes\id_\mathcal{K})\circ\Psi=\Psi\circ(\lambda^A_z\otimes\id_\mathcal{K})$ for all $z\in\mathbb{T}$. As a corollary to Theorem~\ref{thm:4}, we now prove the converse.

\begin{corollary} \label{cor:3}
	Let $A$ and $B$ be finite square $\{0,1\}$-matrices, and assume that every row and every column of $A$ and $B$ is nonzero. There a $*$-isomorphism $\Psi:\Oo_A \otimes\mathcal{K}\to \Oo_B \otimes\mathcal{K}$ such that $\Psi(\mathcal{D}_{A}\otimes\mathcal{C})=\mathcal{D}_{B}\otimes\mathcal{C}$ and $(\lambda^B_z\otimes\id_\mathcal{K})\circ\Psi=\Psi\circ(\lambda^A_z\otimes\id_\mathcal{K})$ for all $z\in\mathbb{T}$ if and only if $(\TSS_A,\tss_A)$ and $(\TSS_B,\tss_B)$ are conjugate.
\end{corollary}

\begin{proof}
As in the proof of Corollary~\ref{cor:2}, let $E_A$ be the graph of $A$, and $E_B$ the graph of $B$. The result then follows from the equivalence of (II) $\iff$ (VI) of Theorem~\ref{thm:4} applied to the graphs $E_A$ and $E_B$.
\end{proof}

\section{Acknowledgements}

The second author is grateful for the provision of travel support from Aidan Sims and from the Institute for Mathematics and its Applications at the University of Wollongong, and grateful for the hospitality of the first author and his family and the University of the Faroe Islands during his visit.

\end{document}